\documentclass[10pt]{amsart}
\usepackage[cp1251]{inputenc}

\usepackage{amsmath}
\usepackage{amssymb}
\usepackage{amsfonts}

\begin{document}
\renewcommand{\refname}{References}
\thispagestyle{empty}

\newtheorem{Theorem}{Theorem}
\newtheorem{Corollary}{Corollary}
\newtheorem{Example}{Example}
\newtheorem{Remark}{Remark}
\newtheorem{Proposition}{Proposition}
\newtheorem{Definition}{Definition}

\def\proofname{Proof}

\title[A Superclass of the Posinormal Operators]{A Superclass of the Posinormal Operators}

\author{Henry Crawford Rhaly Jr.}
\address{Henry Crawford Rhaly Jr.
\newline\hphantom{iii} Jackson, MS  39206, U.S.A.}
\email{rhaly@member.ams.org}

\maketitle {\small
\begin{quote}
\noindent{\sc Abstract. }  The starting place is a brief proof of a well-known result, the hyponormality of $C_k$ (the generalized Ces\`{a}ro operator of order one) for $k \geq 1$.   This leads to the definition of a superclass of the posinormal operators.   It is shown that all the injective unilateral weighted shifts belong to this superclass.

Sufficient conditions are determined for an operator in this superclass to be posinormal and hyponormal.   A connection is established between this superclass and some recently-published sufficient conditions for a lower triangular factorable matrix to be a hyponormal bounded linear operator on $\ell^2$.\medskip

 \end{quote}
}

\vspace{2mm}

\section{Introduction}

\subsection{Preliminaries}\

\vspace{2mm}

If $B(H)$ denotes the set of all bounded linear operators on a Hilbert space $H$, then $A \in B(H)$ is said to be \textit{posinormal} if there exists a positive operator $P \in B(H)$ satisfying $AA^*=A^*PA$.   The operator $A$ is \textit{coposinormal} if $A^*$ is posinormal.  From [\textbf{8}, Theorem 2.1], we know that $A$ is posinormal if and only if \[\gamma^2A^*A \geq AA^*\] for some $\gamma \geq 0$.  $A$ is \textit{hyponormal} when $\gamma=1$.   The operator $A$ is \textit{dominant} if $Ran (A - \lambda) \subset Ran (A - \lambda)^*$ for all $\lambda$ in the spectrum of $A$; $A$ is dominant if and only if $A-\lambda$ is posinormal for all complex numbers $\lambda$ ([\textbf{8}, Proposition 3.5]).  Hyponormal operators are necessarily dominant.   If $A$ is posinormal, then $Ker A \subset Ker A^*$; see [\textbf{8}, Corollary 2.3].

A lower triangular infinite matrix $M =  [m_{ij}] $, acting through multiplication to give a bounded linear operator on  $\ell^2$, is {\it factorable} if its entries are

\[m_{ij} = \left \{ \begin{array}{lll}
a_ic_j & if  &  j \leq i\\
0 & if & j > i \end{array}\right.\]
where $a_i$ depends only on $i$ and $c_j$ depends only on $j$;  the matrix $M$ is {\it terraced} if $c_j = 1$ for all $j$.    

\subsection{The Motivating Example}\

\vspace{2mm}

For fixed $k > 0$, the generalized Ces\`{a}ro matrices of order one are the terraced matrices $C_k$ that occur when $a_i = \frac {1}{k+i}$ for all $i$.   

\begin{Proposition}  $C_k$ is posinormal for all $k>0$, and $C_k$ is hyponormal for all $k \geq 1$. 
\end{Proposition} 

\begin{proof}    If $Q:\equiv diag \{k, 1, 1, 1, ....\}$ and $P:\equiv diag\{\frac{k+i}{k+i+1}: i = 0, 1, 2, 3, ....\}$,  it can be verified that $C_k Q C^*_k =   C^*_k P C_k$ for all $k > 0$.    If $\delta:\equiv max \{\frac{1}{k},1\}$, then for all $f$ in $\ell^2$ we have

 \noindent \hspace{6mm} $\langle (\delta C^*_k C_k - C_kC^*_k) f , f \rangle = \langle (\delta C^*_k C_k - \delta C^*_k P C_k +  \delta C_k Q C^*_k - C_kC^*_k) f , f \rangle$ 

\hspace{37mm} $= \langle  \delta (I-P)C_kf,C_kf \rangle + \langle(\delta Q-I)C^*_kf,C^*_kf \rangle  \geq 0$, 

\noindent and this gives the result. 
\end{proof}

\vspace{2mm} \noindent We note that the preceding proof first appeared in [\textbf{10}]; for different proofs of the hyponormality of $C_k$ for $k \geq 1$, see [\textbf{8}], [\textbf{13}].  The key role played by the relationship $C_kQC_k^*=C_k^*PC_k$ in the proof of Proposition 1 leads us to the definition of a very large class of operators.   

\section{Definition, Properties, and Examples of \\ Supraposinormal Operators}

\begin{Definition} If $A  \in B(H)$, we say that $A$ is supraposinormal if there exist positive operators $P$ and $Q$ on $H$ such that  $AQA^*= A^*PA $, where at least one of $P$, $Q$ has dense range.  It will sometimes be convenient to refer to the ordered pair $(Q,P)$ as an interrupter pair associated with $A$.  
\end{Definition}

It is straightforward to verify that supraposinormality is a unitary invariant.  We note that a normal operator $A$ is supraposinormal with interrupter pair $(I,I)$.  

As the name suggests,  this superclass of operators contains all the posinormal operators (and hence all the hyponormal operators and all the invertible operators; see [\textbf{8}]), as well as all the coposinormal operators:   If $A$ is posinormal, then $AA^* = A^*PA$ for some positive operator $P$, so $A$ is supraposinormal with interrupter pair $(I,P)$.   If $A$ is coposinormal, then $A^*A = AQA^*$ for some positive operator $Q$, so $A$ is supraposinormal with interrupter pair $(Q,I)$.    

\begin{Proposition}  The collection $S$ of all supraposinormal operators on $H$ forms a cone in $B(H)$, and $S$ is closed under involution. 
\end{Proposition}

\begin{proof}  It is easy to see that $S$ is closed under scalar multiplication, so $S$ contains all $\alpha A$ for $A \in S$ and $\alpha \geq 0$, and therefore $S$ is a cone.  Moreover, it is equally easy to see that $A$ is supraposinormal if and only if $A^*$ is supraposinormal, so $S$ is closed under involution.
\end{proof}

\begin{Theorem}  Suppose $A \in B(H)$ satisfies $AQA^*= A^*PA$ for positive operators $P, Q \in B(H)$. 

\noindent (a)  If $Q$ has dense range, then $A$ is supraposinormal and $Ker A \subset Ker A^*$.    

\noindent (b)  If $P$ has dense range, then $A$ is supraposinormal and $Ker A^* \subset Ker A$.    

\noindent (c)  If $Q$ is invertible, then the supraposinormal operator $A$ is posinormal.  

\noindent (d)  If $P$ is invertible, then the supraposinormal operator $A$ is coposinormal.

\noindent (e)  If $P$ and $Q$ are both invertible, then $A$ is both posinormal and coposinormal with 

\hspace{1mm}$Ker A = Ker A^*$ and $Ran A = Ran A^*$.
\end{Theorem}

\begin{proof}  For (a) and (b), the proof is straightforward.  For (c) and (d), the proof requires only a minor adjustment in the discussion of Douglas's Theorem at the beginning of Section 2 in [\textbf {8}].  The proof of (e) is also straightforward.
\end{proof}

\begin{Corollary}  If $A \in B(H)$ is supraposinormal, then $Ker A \subset Ker A^*$ or $Ker A^* \subset Ker A$.
\end{Corollary}

\begin{Corollary}  If $A \in B(H)$ is posinormal with an invertible interrupter, then $A$ is also coposinormal.
\end{Corollary}

\begin{Corollary}  If $A$ is supraposinormal with interrupter pair $(P,P)$ for some positive operator $P$, then $Ker A = Ker A^*$; also, $\sqrt{P} A \sqrt{P}$ is normal.
\end{Corollary}

\begin{Theorem}  Assume $A-\lambda$ is supraposinormal for distinct real values $\lambda = 0, r_1$, and $r_2$, and assume that the same interrupter pair $(Q,P)$ serves $A-\lambda$ in each of those three cases.  Then $Q=P$ and $Ker (A-\lambda) = Ker (A-\lambda)^*$ when $\lambda = 0, r_1$, and $r_2$.
\end{Theorem}

\begin{proof} Since $(A-\lambda)Q(A-\lambda)^* = (A-\lambda)^*P(A-\lambda)$ for $\lambda = 0, r_1$, and $r_2$, we find that for $k = 1$ and $2$, \[(A-r_k)Q(A-r_k)^* = (A-r_k)^*P(A-r_k)\] reduces to \[PA + A^*P + r_k Q = QA^* + AQ + r_k P.\]   Thus $(r_1 - r_2) Q = (r_1 - r_2) P$, so $Q=P$.  The assertion that $Ker (A-\lambda) = Ker (A-\lambda)^*$ for $\lambda = 0, r_1$, and $r_2$ follows from Corollary $3$.
\end{proof}

\begin{Definition}  For $A \in B(H)$, we say that $A$ is totally supraposinormal if $A - \lambda$ is supraposinormal for all complex numbers $\lambda$.
\end{Definition}

\begin{Theorem}  If $A \in B(H)$ is totally supraposinormal and the same two positive operators $Q,P \in B(H)$  form an interrupter pair $(Q,P)$ for $A - \lambda$ for all complex numbers $\lambda$, then $Q = P$; it also follows that $Ker (A - \lambda) = Ker (A - \lambda)^*$ for all $\lambda$.
\end{Theorem}

\begin{proof}  This result is a consequence of Theorem $2$ and Corollary $3$.
\end{proof}

We have already observed that the class of posinormal operators is included in the class of supraposinormal operators, which is included in the class of all operators.  As we are about to see, the unilateral weighted shifts are enough to show that these inclusions are all proper.

Let $\{e_n\}$ denote the standard orthonormal basis for $\ell^2$.  

\begin{Proposition}  A unilateral weighted shift $W$ with the weight sequence 

\noindent $\{w_n : w_0 \neq 0\}_{n\geq0}$ is supraposinormal if and only if it is injective.
\end{Proposition}

\begin{proof}  (1)  First assume that $W$ is injective, so $w_n \neq 0$ for all $n$.  If \[Q :\equiv diag \{|w_1|^2, |w_2|^2, |w_3|^2, . . . .\}\] and \[P:\equiv diag \{p_0, 0, |w_0|^2, |w_1|^2, |w_2|^2,  . . . . \}\] for some $p_0 > 0$, then it is straightforward to very that $WQW^* = W^*PW$.   Since $\{|w_1|^2, |w_2|^2, |w_3|^2, . . .\}$ is a strictly positive sequence, $Q$ is a one-to-one positive operator; thus $Q$ has dense range, so $W$ is supraposinormal.   (2) Next we assume that $W$ is not injective, so $w_n =  0$ for some $n > 0$; assume that $n$ is the smallest integer for which this holds.  We have $We_n = 0$ and $W^*e_n = \overline{w_{n-1}}e_{n-1} \neq 0$, so $Ker W \not\subset Ker W^*$.   Also, $W^*e_0 = 0$ while $We_0 = w_0 e_1 \neq 0$, so $Ker W^* \not\subset Ker W$.  It follows from Corollary $1$ that $W$ is not supraposinormal.
\end{proof}

\begin{Corollary}  Every injective unilateral weighted shift is supraposinormal.
\end{Corollary}

We note that a noninjective unilateral weighted shift can also be supraposinormal, as the next example illustrates.

\begin{Example}  Let $W$ denote the unilateral weighted shift with $w_0 = 0$ and $w_n = 1$ for all $n \geq 1$.  Glancing at the proof of Proposition $3$, we take $Q =  I$ and $P := diag \{1, 0, 0, 1, 1, 1, 1, . . . .\}$.  One easily verifies that $ WW^* = diag \{0, 0, 1, 1, 1, 1, . . . .\} = W^*PW$, so $W$ is posinormal and hence supraposinormal. 

\end{Example}

In fact, it can be shown that if, for some nonnegative integer $n$, $W$ satisfies  

\begin{center} ($1$) $w_k = 0$ for $0 \leq k \leq n$ and $w_k \neq 0$ for $k > n$ and \end{center}

\hspace{18mm} ($2$)  $\sup_{k > n} |\frac{w_k}{w_{k+1}}| < +\infty$, 

\noindent then $W$ is posinormal and hence also supraposinormal.

\begin{Proposition}  An injective unilateral weighted shift is posinormal if and only if $\sup_n |\frac{w_n}{w_{n+1}}| < +\infty$.
\end{Proposition}

\begin{proof}  See [\textbf{4}].  
\end{proof}

Proposition $3$  gives us a collection of operators that are supraposinormal, as well as a collection of operators that are not supraposinormal.  We emphasize the latter now with a specific example.

\begin{Example}  Suppose $W$ is the unilateral weighted shift with weights $w_{2n}=1$ and $w_{2n+1}=0$ for all $n$.    It follows from Proposition $3$ that $W$ cannot be supraposinormal.
\end{Example}  

Next we present an example of a supraposinormal operator that is nether posinormal nor coposinormal.

\begin{Example}  Let $W$ denote the unilateral weighted shift with weights $w_{2n}=1$ and $w_{2n+1}=1/n$ for all $n$.   By Proposition $3$, $W$ is supraposinormal.  Since $\sup_n |\frac{w_n}{w_{n+1}}| = +\infty$, it follows from Proposition $4$ that $W$ is not posinormal.  Since $W^*e_0 = 0$ while $We_0 = e_1$, we see that $Ker W^* \not\subset Ker W$, so $W$ is also not coposinormal.   
\end{Example}

We model the following proof on that for the motivating example.

\begin{Theorem}  Assume that the bounded linear operator $A$ on $H$ is supraposinormal with $AQA^*= A^*PA $.   If  $D$ is a positive operator satisfying  

\noindent (1) \hspace{20mm} $\delta_1Q \geq D \geq \delta_2 P \geq 0$ \hspace{5mm} for some constants $\delta_1, \delta_2 >0$,

\noindent  then $\sqrt{D}A\sqrt{D}$ is posinormal.   If (1) holds for some pair of positive constants $\delta_1, \delta_2$ with $\delta_1\leq \delta_2$, then $\sqrt{D}A\sqrt{D}$ is hyponormal.
\end{Theorem}

\begin{proof}  We find that 

\noindent $\langle [\frac{\delta_1}{\delta_2}(\sqrt{D}A^*\sqrt{D})(\sqrt{D}A\sqrt{D})-(\sqrt{D}A\sqrt{D})(\sqrt{D}A^*\sqrt{D})]f,f\rangle =$

\noindent $\langle [\frac{\delta_1}{\delta_2}\sqrt{D}A^*DA\sqrt{D} - \delta_1\sqrt{D}A^* P A\sqrt{D} + \delta_1\sqrt{D}A Q A^*\sqrt{D}- \sqrt{D}ADA^*\sqrt{D}]f,f \rangle =$ 

\noindent $\delta_1 \langle (\frac{1}{\delta_2}D-P)A\sqrt{D}f,A\sqrt{D}f \rangle + \langle(\delta_1Q-D)A^*\sqrt{D}f,A^*\sqrt{D}f \rangle  \geq 0$ 

\noindent for all $f$ in $\ell^2$, as needed.
\end{proof}

\begin{Corollary}   If $A$ is supraposinormal operator on $H$ with $AQA^*= A^*PA $ and  

\noindent (2) \hspace{20mm} $\delta_1Q \geq I \geq \delta_2P \geq 0$  \hspace{5mm}  for some constants $\delta_1, \delta_2 >0$,   

\noindent then $A$ is posinormal.  If (2) holds for some pair of positive constants $\delta_1, \delta_2$ with $\delta_1\leq \delta_2$, then $A$ is hyponormal.
\end{Corollary}

We note that Theorem $4$ and Corollary $5$ above are restricted to the case where $P$ is dominated by a multiple of $Q$.

In the next section we will apply this theorem and its corollary to factorable matrices.    We note that a large collection of examples of supraposinormal factorable matrices $M$ having interrupter pair $(I,P)$ with $P$ diagonal and $I \geq P$, in which case $M$ is hyponormal, can be found in [\textbf{11}].

\section{Application to Factorable Matrices}

Throughout this section we will restrict our attention to those factorable matrices $M$ that are lower triangular and give bounded linear operators on $H = \ell^2$.

\subsection{Sufficient Conditions for Hyponormality of a Factorable Matrix}

\begin{Theorem}   If $M$ is a supraposinormal factorable matrix with $MQM^*= M^*PM$, and if $P$, $Q$, and $D$ are diagonal matrices satisfying \[Q \geq D \geq P \geq 0,\] then $\sqrt{D}M\sqrt{D}$ is a hyponormal factorable matrix.
\end{Theorem}

\begin{proof}  This is an immediate consequence of Theorem 4.
\end{proof}

\begin{Example}  (Generalized Ces\`{a}ro matrices of order one for $k \geq 1$) We have already seen that $C_k$ is hyponormal for $k \geq 1$ and that $C_k$ is posinormal for all $k>0$.  $C_k$ is known to be non-hyponormal when $0 < k < 1$ (see [\textbf{$8$}]).   If $k \geq 1$ and $D:\equiv diag \{d_i : i = 0, 1, 2, 3, ....\}$ where $\frac{k}{k+1} \leq d_0 \leq k$ and $\frac{k+i}{k+i+1} \leq d_i \leq 1$ for $i = 1, 2, 3, ....$, then Proposition $1$ (proof) and Theorem $5$ together guarantee that $\sqrt{D}C_k\sqrt{D}$ is another hyponormal factorable matrix.   We note that Theorem $1$(d) can be used to prove that $C_k$ is coposinormal for all $k > 0$.  
\end{Example}

The following proposition will be useful to us throughout the remainder of this section.

\begin{Proposition}  Assume that the factorable matrix $M=M(\{a_i\},\{c_j\})$ is a bounded operator on $\ell^2$ with $a_i , c_j > 0$ for all $i, j$, $\{\frac{a_k}{c_k}\}$ is strictly decreasing to $0$.  If \[P :\equiv diag \{\frac{c_{k+1}a_k - c_ka_{k+1}}{c_kc_{k+1}a_k^2}: k = 0, 1, 2, ....\} \in B(\ell^2),\] \[Q :\equiv diag \{\frac{1}{c_0a_0}, \frac{c_{k+1}a_k - c_ka_{k+1}}{c_{k+1}^2a_ka_{k+1}}: k = 0, 1, 2, ....\}\in B(\ell^2),\] then $M$ is supraposinormal with interrupter pair $(Q , P)$.
\end{Proposition}

\begin{proof}  Once the hypothesis is assumed, it is straightforward to verify that \[M^*PM =  \left(\begin{array}{ccccc}
c_0a_0 & c_0a_1 & c_0a_2 & c_0a_3  & \ldots\\
c_0a_1 & c_1a_1 & c_1a_2 & c_1a_3  & \ldots\\
c_0a_2 & c_1a_2 & c_2a_2 & c_2a_3  & \ldots\\
c_0a_3 & c_1a_3 & c_2a_3 & c_3a_3  & \ldots\\
\vdots & \vdots & \vdots& \vdots & \ddots \end{array}\right) = MQM^*.\]  Clearly the positive operators $P$ and $Q$ are one-to-one, so they both have dense range; thus $M$ is supraposinormal with interrupter pair $(Q , P)$
\end{proof}

Our goal now is to obtain restrictions on the sequences $\{a_i\}, \{c_j\}$ that are sufficient to guarantee that the factorable matrix $M :\equiv M (\{a_i\},\{c_j\}) \in B(\ell^2)$ is hyponormal.       

\begin{Theorem}  Assume that $M=M(\{a_i\},\{c_j\})$ is a bounded operator on $\ell^2$ with $a_i , c_j > 0$ for all $i, j$, $\{\frac{a_k}{c_k}\}$ is strictly decreasing to $0$, and $\{\frac{a_kc_k}{a_{k+1}c_{k+1}}\}$ is bounded.   If there exists a $\delta > 0$ such that $0 < c_0a_0 \leq \delta$ and $\frac{1}{\delta}\frac{a_k}{c_k}(\delta-c_ka_k)\leq \frac{a_{k+1}}{c_{k+1}} \leq \frac{\delta \frac{a_k}{c_k}}{\delta+c_{k+1}^2\frac{a_k}{c_k}}$ for each nonnegative integer $k$, then $M$ is hyponormal.
\end{Theorem}

\begin{proof}  Assume the conditions of the hypothesis and also that \[sup \{\frac{c_{k+1}a_k - c_ka_{k+1}}{c_kc_{k+1}a_k^2}: k = 0, 1, 2, ....\} < \infty; \hspace{20mm} (*)\] this is enough to guarantee that  $sup \{\frac{c_{k+1}a_k - c_ka_{k+1}}{c_{k+1}^2a_ka_{k+1}}: k = 0, 1, 2, ....\} < \infty$ also.  \\ If $P$ and $Q$ are the operators defined in Proposition $5$, then $M^*PM=MQM^*$.   By Corollary $5$, $M$ will be hyponormal if there exists a $\delta > 0$ such that $\delta Q \geq  I \geq \delta P > 0$, and this leads to inequalities stated in the theorem.  We note that since $P \leq (1/\delta) I$, the assumption (*) was not needed in the hypothesis.
\end{proof}

\begin{Example}  (Toeplitz Matrix)  Suppose that $M$ is the factorable matrix with entries $m_{ij}=a_ic_j$ where $a_i=r^i$, $c_j=1/r^j$ for all $i,j$ where $0<r<1$.      One easily verifies that Theorem $6$ is satisfied with $\delta = 1/(1-r^2)$, so $M$ is hyponormal.
\end{Example}

We note that the conditions presented in Theorem $6$ are not necessary for the hyponormality of a factorable matrix.   For consider the case when $c_j = 1$ for each $j$ and $a_i= (i+3)/(i+2)^2$ for each $i$.   This example is known to be hyponormal since it satisfies the hypothesis of [\textbf{9}, Theorem 2.2], but it does not satisfy the inequality in Theorem 6 since that would require $1 \geq \delta \geq 12/11$, an impossibility.   

We point out that for $\delta = 1$, the inequality in Theorem $6$ reduces to the result in [\textbf{12}], which was obtained using a somewhat different approach, without invoking posinormality; for several examples that hold for $\delta = 1$, see that paper.

\subsection{Some Non-Hyponormal Examples}\

\vspace{2mm}

We now investigate posinormality and coposinormality for some non-hyponormal, supraposinormal factorable matrices.

\begin{Example}  (Fibonacci matrix; see [\textbf{3}])   We recall that the Fibonacci sequence $\{f_n\}$ is defined by the linear recurrence equations 
$f_0 = 0$, $f_1 = 1$, and 
\begin{center} $f_n = f_{n-1} + f_{n-2}$ \hspace{6mm} for $n \geq 2$.  \end{center} Let $M$ denote the factorable matrix with nonzero entries 

\vspace{2mm}
\hspace{5mm} $m_{ij}=a_ic_j$ where $a_i = 1/(f_{i+1}f_{i+2})$ and $c_j = f_{j+1}^2$ for $j \leq i$.  
\vspace{2mm}

\noindent   When $n \geq 3$,  \[c_na_1\prod_{j=2}^{n-1} (c_0a_0-c_ja_j) /( c_0^na_0^n) = f_{n+1}/f_n > 1,\] so it follows from [\textbf{12}, Theorem $2$] that $M$ is not a dominant operator and is hence also not hyponormal.  If \[P :\equiv  diag \{(f_{n+2}^2f_{n+3}-f_{n+1}^3)/(f_{n+1}f_{n+2}f_{n+3}) : n = 0, 1, 2, 3, ....\},\] and 
\[Q :\equiv  diag \{1, (f_{n+2}^2f_{n+3}-f_{n+1}^3)/f_{n+2}^3 : n = 0, 1, 2, 3, ....\}\] \hspace{24mm} $= diag \{(f_{n+1}^2f_{n+2}-f_{n}^3)/f_{n+1}^3 : n = 0, 1, 2, 3, ....\}$, 
\vspace{2mm}

\noindent it follows from Proposition $5$ that $M^*PM = MQM^*$.  It can be verified that  \[I \leq Q \leq 2I \textrm{ \hspace {10mm} and  \hspace {10mm}} (1/2)I \leq P \leq 4I;\] thus \[Q \geq I \geq (1/4)P,\] so by Corollary $5$, $M$  is posinormal.  Similarly, since \[2P \geq I \geq (1/2)Q,\] $M$ is also coposinormal.
\end{Example}

\begin{Example}  ($q$-Ces\`{a}ro matrix for $q > 1$; see [\textbf{1}], [\textbf{14}])   If $M$ is the factorable matrix with nonzero entries 

\hspace{5mm} $m_{ij}=a_ic_j$ where $a_i= (q-1)/(q^{i+1}-1)$ and $c_j=q^j$ for $0 \leq j \leq i$,

\noindent and if \[P :\equiv  diag \{(q^{n+1}-1)(q^{n+2}+q^{n+1}-1)/[q^{n+1}(q^{n+2}-1)] : n = 0, 1, 2, 3, ....\}\] and \[Q :\equiv diag \{1, (q^{n+2}+q^{n+1}-1)/q^{n+2}: n = 0, 1, 2, 3, ....\}\] 

\hspace{20mm} $= diag \{ 1+ 1/q - 1/q^{n+1}: n = 0, 1, 2, 3, ....\}$,
\vspace{1mm}

\noindent it follows from Proposition $5$ that $MQM^* = M^*PM$.  It is straightforward to check that  \[Q \geq I \geq (1/2) P \textrm{ \hspace {10mm} and  \hspace {10mm}} (q+1) P \geq I \geq [1/(q+1)]Q,\]  so we know from Corollary $5$ that $M$ is both posinormal and coposinormal for $q > 1$.  It is demonstrated in [\textbf{12}] that $M$ is not dominant and hence also not hyponormal.
\end{Example}

\begin{Example}  ($q$-Ces\`{a}ro matrix for $0 < q < 1$; see [\textbf{2}] , [\textbf{14}])   If $M$ is the factorable matrix with nonzero entries 

\hspace{5mm} $m_{ij}=a_ic_j$ where $a_i=(1-q)q^i/(1-q^{i+1})$ and $c_j= 1/q^j$ for $0 \leq j \leq i$, 

\noindent and if \[P:\equiv  diag \{(1-q^{n+1})(1+q-q^{n+2})/(1-q^{n+2}) : n = 0, 1, 2, 3, ....\}\] and \[Q :\equiv diag \{1, 1+q-q^{n+2}: n = 0, 1, 2, 3, ....\}\] 

\hspace{29mm} $= diag \{1+q-q^{n+1}: n = 0, 1, 2, 3, ....\}$,
\vspace{1mm}
 
 \noindent it follows from Proposition $5$ that $MQM^* =  M^*PM$.   One can easily check that  \[Q \geq I \geq (1/2) P \textrm{ \hspace {10mm} and  \hspace {10mm}} [(q+1)/q] P \geq I \geq [q/(q+1)]Q,\]  so $M$ is both posinormal and coposinormal for $0 < q <1$.  We know from [\textbf{12}] that $M$ is not dominant and not hyponormal.
\end{Example}

\section{Conclusion}

We close with two questions that seem natural but whose answers have not been found here:  (1) Does there exist an operator $A$ that is totally supraposinormal but neither dominant nor codominant?  (2) Does there exist an operator $A$ that is totally supraposinormal but neither posinormal nor coposinormal?

\end{document}